%% file: main.tex
\documentclass{article}
\usepackage{amsmath,amsthm,amssymb}
\usepackage[author-year]{amsrefs}
\usepackage{mathtools}          
\usepackage[inline]{enumitem}   

\usepackage{pgfplots}           
\pgfplotsset{compat=1.4}

\usepackage[pdftex]{hyperref}
\hypersetup{
    breaklinks,
    colorlinks,
    linkcolor={red!50!black},
    citecolor={blue!50!black},
    urlcolor={blue!80!black}
}

\input{macros.tex}

\title{The Lorenz Renormalization Conjecture}
\author{Bj\"orn Winckler\thanks{%
    Department of Mathematics, Imperial College London, London SW7 2AZ, UK,
    \texttt{bjorn.winckler@gmail.com}.
    This project has received funding from the European Union's Horizon 2020
    research and innovation programme under the Marie Sk\l{}odowska-Curie grant
    agreement No~743959.}}

\begin{document}

\maketitle

\begin{abstract}
    The renormalization paradigm for low-dimensional dynamical systems is that
    of hyperbolic horseshoe dynamics.
    Does this paradigm survive a transition to more physically relevant systems
    in higher dimensions?
    This article addresses this question in the context of Lorenz dynamics
    which originates in homoclinic bifurcations of flows in three dimensions
    and higher.
    A conjecture classifying the dynamics of the Lorenz renormalization
    operator is stated and supported with numerical evidence.
\end{abstract}

\section{Introduction}

Renormalization in low-dimensional dynamical systems is characterized by
hyperbolic horseshoe dynamics with
contraction within topological families and expansion otherwise.
There are an abundance of low-dimensional systems which adhere to this
paradigm, such as unimodal maps \cite{AL11}, critical circle maps
\cite{Y03} and circle maps with breaks \cite{KT13}; as well as partial
results for dissipative H\'enon-like maps \cite{CLM05}, area-preserving maps
\cites{EKW84,GJM16} and higher-dimensional analogs of unimodal maps
\cite{CEK81}.
This research springs from the question: in what way does the renormalization
paradigm need to be modified as its scope is expanded to include more
physically relevant systems coming from flows and maps in higher dimensions?

We expect renormalization phenomena like universality to survive due to the
fact that they have been measured in real physical systems
\cites{ML79,L81}, as first predicted to be possible by \ocite{CT78}.
Surprisingly, it was shown in \ocite{MW16} that even in the
one-dimensional setting of Lorenz maps, instability of renormalization is not
only associated with changes in topology;
the dynamics of the renormalization operator inside topological classes is not
necessarily a contraction.
This also has a fundamental impact on the question of rigidity as
discussed in \sref{sec:rigidity}.

The purpose of this article is to state a conjecture which classifies the
dynamics of the Lorenz renormalization operator and to support this conjecture
with numerical experiments.
We hope that it will act as a focus for what should aim to be proven for these
systems.
More importantly, we wish to provide an indication of what kind of
renormalization phenomena to expect as the field transitions towards
physically relevant systems.

The article is organized into two sections.
In this introduction we go over the necessary definitions and make several
remarks along the way before stating the Lorenz Renormalization Conjecture in
\sref{sec:renorm-conj}.
Having accomplished that, we go on to describe the numerical experiments
performed to support the conjecture and include the results of these
experiments.
The source code, together with instructions on how to reproduce the results,
are freely available online \cite{W18}.

\subsection{Lorenz maps}

\begin{definition} \label{def:lorenz}
    Let $I = [l, r]$ be a closed interval.
    A \defn{Lorenz map} $f$ on $I$ is a monotone increasing function which is
    continuous except at a \defn{critical point}, $c \in (l,r)$, where it has a
    jump discontinuity, and $f(I\setminus\{c\}) \subset I$
    (see figure~\ref{fig:monotone-8-2}).

    The branches\footnote{%
        Even though $f$ is undefined at $c$, its branches continuously extend
        to $c$ since $f$ is bounded.}
    $f_0:[l,c] \to I$ and $f_1:[c,r] \to I$ of $f$ are assumed to satisfy:
    \begin{enumerate*}[label=(\roman*)]
        \item $f_0(c) = r$ and $f_1(c) = l$,
        \item $f_k(x) = \phi_k(\abs{c - x}^\alpha)$, for some
            \defn{critical exponent} $\alpha > 0$, and
            $\Cset^2$--diffeomorphisms $\phi_k$, $k=0,1$.
    \end{enumerate*}

    The \defn{set of Lorenz maps} on $[0,1]$ is denoted $\Lset$.
\end{definition}

\begin{convention}
    Unless the interval $I$ in the above definition is mentioned, it is
    implicitly assumed to be the unit interval $[0,1]$.
\end{convention}

\begin{remark}
    It bears pointing out that the critical point $c$ is \emph{not} fixed, but
    depends on the map $f$.
    Later on we will see that the critical point moves under renormalization.
    This is an essential feature of Lorenz maps which has very strong
    consequences on the dynamics and results in new renormalization phenomena
    not present in unimodal and circle dynamics \cite{MW16}.
\end{remark}

\begin{remark}
    The second condition on the branches ensures that the behavior of $f$ near
    the critical point is like that of the power map $x^\alpha$ near $0$.
    This condition and the assumption $\alpha > 1$ leads to a well-defined
    renormalization theory.
\end{remark}

\begin{convention}
    The critical exponent $\alpha \in \reals$ is fixed and $\alpha > 1$.
\end{convention}

\begin{remark} \label{rem:flow}
    Lorenz maps were introduced by \ocite{GW79} in order to describe the
    dynamics of three-dimensional flows geometrically similar to the
    well-known Lorenz system \cite{L63}.
    The flows they consider have a saddle with a one-dimensional unstable
    manifold which exhibits recurrent behavior.
    Their construction is to take a transversal section to the stable manifold
    and assume that the associated first-return map has an invariant foliation
    whose leaves are exponentially contracted.
    Taking a quotient over the leaves results in a one-dimensional map as
    described by definition~\ref{def:lorenz}.

    In the above construction the critical exponent $\alpha$ naturally comes
    out as the absolute value of the ratio between two eigenvalues of the
    linearized flow at the singularity.
    In particular, it is important for Lorenz theory to be able to handle any
    real critical exponent $\alpha > 0$ (as opposed to unimodal theory where it
    may be possible to get away with saying something like ``the critical
    exponent is generically two'').
    \ocite{GW79} considered $\alpha \in (0,1)$; the first to investigate
    $\alpha > 1$ were \ocite{ACT81}.
\end{remark}

\begin{remark}
    In more generality, Lorenz maps can be thought of as the underlying
    dynamical model for a large class of higher dimensional flows undergoing a
    homoclinic bifurcation.
    Hence there are very strong reasons why Lorenz dynamics needs to be further
    explored.
    We can only guess that this theory is still so largely underdeveloped, as
    compared to unimodal and circle dynamics, because of the fact that the
    holomorphic tools developed in these other theories are not suitable for
    adaptation to discontinuities and arbitrary real critical exponents.
    New ideas and tools are desperately needed!
\end{remark}

\begin{remark}
    There is a genuine problem relating to smoothness that needs mentioning.
    Even if the invariant foliation mentioned in remark~\ref{rem:flow} is
    smooth, the holonomy map need not be \cites{M97,HPS77}.
    Hence, the associated Lorenz map need not have $\Cset^2$ branches,
    regardless of how smooth the initial flow is.
    Without $\Cset^2$--smoothness the renormalization apparatus breaks down
    \cite{CMdMT09}.
    In transferring results about maps to flows this problem needs to be
    addressed.
\end{remark}

\subsection{Renormalization}

\begin{definition} \label{def:rescale}
    Let $A_I:[0,1] \to I$ denote the increasing affine map taking $[0,1]$
    onto~$I$.
    The \defn{rescaling} to $[0,1]$ of $g: U \to V$ (synonymously,
    $g$ \defn{rescaled} to $[0,1]$) is the map $G:[0,1]\to[0,1]$ defined by
    $G = A^{-1}_V \circ g \circ A_U$.
    In this situation we also conversely say that $g$ is a rescaling of $G$.
\end{definition}

\begin{definition} \label{def:renorm}
    A Lorenz map $f$ is \defn{renormalizable} iff there exist $n_0,n_1 \geq 2$
    such that $I = [f^{n_1 - 1}(0), f^{n_0 - 1}(1)]$ is contained in $(0,1)$ and
    contains $c$ in its interior, and such that the first-return map to $I$ is
    again a Lorenz map (on~$I$); the first-return map rescaled to $[0,1]$ is
    called a \defn{renormalization} of~$f$ and the symbolic coding of its
    branches defines the \defn{type} (or \defn{combinatorics}),
    $\word = (\word_0,\word_1)$, of the renormalization.\footnote{%
        Explicitly, let $I_k = I \cap [k,c)$ and define $\word_k$ to be the
        finite word on symbols $\{0,1\}$ such that
        $f^j(I_k) \subset [\word_k(j),c)$ for $j=0,\dotsc,\abs{\word_k}-1$ and
        $k=0,1$.}
    In this case we also say that $f$ is \defn{$\word$--renormalizable} and
    call the rescaled first-return map a \defn{$\word$--renormalization}.
\end{definition}

\begin{definition}
    The type $\word = (\word_0,\word_1)$ is said to be of
    \defn{monotone combinatorics} if $\word_0 = 011\dotsm1$ and
    $\word_1 = 100\dotsm0$;
    more succinctly, it is also called \defn{$(a,b)$--type}, where
    $a = \abs{\word_0} - 1$ and $b = \abs{\word_1} - 1$.
\end{definition}

\begin{remark}
    A Lorenz map may have more than one renormalization, but each will have a
    distinct type; in particular, if $f$ is both $\word$--renormalizable and
    $\word'$--renormalizable (with $\word \neq \word'$), then $\word'_0$ and
    $\word'_1$ are finite words on symbols $\{\word_0,\word_1\}$ with at least
    one of each symbol, or vice versa.
    Defining $\abs{\word} = \abs{\word_0} + \abs{\word_1}$ we have that either
    $\abs{\word} < \abs{\word'}$, or $\abs{\word'} < \abs{\word}$ \cite{MdM01}.
\end{remark}

\begin{definition} \label{def:Rop}
    Define the \defn{renormalization operator}, $\Rop$, by sending a
    renormalizable $f$ to the $\word$--renormalization of $f$ for which
    $\abs{\word}$ is minimal.

    Maps for which $\Rop^j f$ is renormalizable for every $j \geq 0$ are called
    \defn{infinitely renormalizable};
    in the special case where $\Rop^j f$ is $\word$--renormalizable and $\word$
    does not depend on $j$, $f$ is called
    \defn{infinitely $\word$--renormalizable}
    (this is also known by the name \defn{stationary combinatorics}).
    The orbit $\{f, \Rop f, \Rop^2 f,\dotsc\}$ is called the
    \defn{successive renormalizations} of~$f$.
\end{definition}

\begin{conjecture}
    The closure of the post-critical set, $\cantor_f$, of an infinitely
    $\word$--renormalizable map $f$ is a minimal Cantor attractor.
\end{conjecture}

\begin{remark}
    For Lorenz maps, $\cantor_f$ is the union of the $\omega$--limit sets of
    the critical values, $f_0(c)$ and $f_1(c)$.
    This conjecture is a theorem for a large class of monotone combinatorics
    \cites{MW14,MW16}.
\end{remark}

\subsection{Rigidity}
\label{sec:rigidity}

\begin{conjecture} \label{conj:Tset}
    The set $\Tset_\word$ of infinitely $\word$--renormalizable Lorenz maps
    coincides with the topological conjugacy class of any $f \in \Tset_\word$.
    Furthermore, $\Tset_\word \subset \Lset$ is a manifold of codimension two.
\end{conjecture}

\begin{remark}
    The first statement would follow if it were shown that there are no
    wandering intervals for $f \in \Tset_\word$.
    This is known for a large class of monotone combinatorics \cites{MW14,MW16}
    but the
    general problem of when Lorenz maps do not support wandering intervals is
    still wide open.
    The codimension of $\Tset_\word$ must be two since topologically full
    families of Lorenz maps are two-dimensional \cite{MdM01}.
\end{remark}

\begin{definition}
    The (classical) notion of \defn{rigidity} is when two topologically
    conjugate maps are automatically smoothly conjugate on their attractors.
\end{definition}

\begin{remark}
    Smooth maps look affine on small scales, so in the presence of rigidity two
    maps have attractors which on a large scale may look very different but
    when zoomed in on a particular spot they start to look the same.
    In this sense rigidity is a strong form of \defn{metric universality}; we
    will not say more about the latter here and instead focus on the former.
\end{remark}

\begin{remark}
    Two crucial ingredients in proving classical rigidity is first to prove
    that successive renormalizations converge and then to control the rate
    of convergence.
    Typically, these ingredients come from the fact that there is a hyperbolic
    renormalization fixed point which attracts both maps.

    It is worth pointing out that the study of rigidity in dynamics was
    initiated by \ocite{H79}, answering a conjecture by \ocite{A61}, but the
    close connection between rigidity and renormalization was only later
    realized.
\end{remark}

\begin{definition}
    The \defn{rigidity class} of $f \in \Tset_\word$ is defined as the set of
    $g \in \Tset_\word$ such that $f$ and $g$ are smoothly conjugate on their
    attractors.
\end{definition}

\begin{remark}
    With this terminology we may characterize classical rigidity as the
    statement that a topological class coincides with a rigidity class.
    From \ocite{MW16} we know that $T_\word$ may, depending on
    $\word$, consist of more than one rigidity class.
    Hence, the classical concept of rigidity is too restrictive, see
    also \ocite{MP17}.
    Instead, the correct notion should be to describe the arrangement of a
    topological class into rigidity classes \cite{MPW17}.

    Even in the classical cases of critical circle maps and unimodal maps there
    is already a natural foliation into codimension--$1$ rigidity classes
    determined by a fixed value for the critical exponent.
    This is however a trivial observation compared to the above mentioned
    articles which concern far more subtle phenomena.
\end{remark}

\subsection{Main conjecture}
\label{sec:renorm-conj}

\begin{definition}
    The successive renormalizations of $f$
    are \defn{attracted to a degenerate flipping $2$--cycle} iff
    $\Rop^{2k} f$ and $\Rop^{2k+1} f$ converge to smooth maps on~$[0,1]$,
    and the critical points have limits
    $c(\Rop^{2k} f) \to 0$ and $c(\Rop^{2k+1} f) \to 1$ (or vice versa).
\end{definition}

\begin{remark}
    Here ``degenerate'' refers to the limits not being Lorenz maps and
    ``flipping'' refers to the fact that the critical points $c(\Rop^k f)$
    flip between being close to zero and being close to one.
    Informally, the limiting cycle can be thought of as two Lorenz maps with 
    critical point $0$ and $1$, respectively.
\end{remark}

\begin{lrc}
    Let $\Tset_\word$ be the set of infinitely $\word$--renormalizable
    Lorenz maps.
    For each $\word$ (such that $\Tset_\word \neq \emptyset$) exactly one
    of the following statements
    holds, and conversely, to each statement there are $\word$ for which it is
    realized:
    \begin{enumerate}[label=(\Alph*)]
        \item \label{T-rigid}
            $\Tset_\word$ is a rigidity class and the stable manifold
            of a hyperbolic renormalization fixed point.
        \item \label{T-foliated}
            $\Tset_\word$ is foliated by codimension--$1$ rigidity
            classes, one of which is the stable manifold of a hyperbolic
            renormalization fixed point.
            The successive renormalizations of any $f \in \Tset_\word$ not in
            this stable manifold are attracted to a degenerate flipping
            $2$--cycle.
        \item \label{T-stratified}
            There exists a nonempty, open and connected set
            $\Tset_\word^\star \subsetneq \Tset_\word$ which is a
            rigidity class as well as the stable manifold of a hyperbolic
            renormalization fixed point;
            its complement,
            $\Tset_\word \setminus \Tset_\word^\star$,
            consists of two connected components which are foliated by rigidity
            classes of codimension one.
            The boundary of $\Tset_\word^\star$ in $\Tset_\word$ is
            a rigidity class as well as the stable manifold of a hyperbolic
            renormalization periodic point of (strict) period two.
            The successive renormalizations of any
            $f \in \Tset_\word \setminus \Tset_\word^\star$
            not in this stable manifold are attracted
            to a degenerate flipping $2$--cycle.
    \end{enumerate}
\end{lrc}

\begin{remark}
    The Lorenz Renormalization Conjecture can be generalized from stationary to
    periodic combinatorics in the obvious way.
    For unbounded combinatorics it is not clear what the right conjecture
    should be as it is possible to force successive renormalizations to not be
    relatively compact by choosing larger and larger return times for one
    branch.
    This leads to Lorenz maps whose attractor does not have a physical measure
    \cite{MW18}.
\end{remark}

\begin{remark}
    A very surprising feature of Lorenz maps is that the dimension of the
    unstable manifold of a renormalization fixed point depends on the
    combinatorics; in cases \ref{T-rigid} and~\ref{T-stratified} the dimension
    is two and in case~\ref{T-foliated} it is three.
    Two of the unstable directions are always related to moving the two
    critical values;\footnote{%
        Just as the one unstable direction for unimodal renormalization is
        related to moving the one critical value.}
    a third unstable direction is gained when the movement of the critical
    point under renormalization becomes unstable (see
    figure~\ref{fig:eigenvalues}).
    In the confounding case~\ref{T-stratified} there is a mix of both:
    the fixed point has two unstable directions, whereas the period--$2$ point
    has three unstable directions.
    This situation occurs e.g.\ for monotone $(8,2)$--type (see
    figure~\ref{fig:monotone-8-2}).
\end{remark}

\begin{remark}
    Evidence for case~\ref{T-rigid} is supported by \ocite{MW14}.
    More recent is \ocite{MW16} where the unstable behavior of the
    renormalization operator within topological classes was discovered; it
    supports case~\ref{T-foliated}.
    Case~\ref{T-stratified} is so far only supported by this article.
    Numerically no other cases seem to occur, see \sref{sec:results} for
    examples of each case.
\end{remark}

\begin{remark}
    Fixed points, $f$, of monotone $(a,a)$--type are symmetric\footnote{%
        That is, the critical point is $c(f)=0.5$ and $1 - f(x) = f(1-x)$.}
    and they are in one-to-one correspondence with unimodal renormalization
    fixed points; it is an exercise to verify that the unimodal map
    $g(x) = f(\min\{x,1-x\})$, with $g(0.5) = 1$, is a fixed point of the
    unimodal renormalization operator.
    In particular, the monotone $(1,1)$--type Lorenz renormalization fixed
    point corresponds to the well known fixed point of the unimodal
    period-doubling operator.

    It seems reasonable to expect all of these ``unimodal fixed points''
    to be dynamically similar, but curiously they are not;
    conjecturally, for $a > \max\{2\alpha - 1, 2\}$ they belong to
    case~\ref{T-foliated}, else they belong to case~\ref{T-rigid}.
    For example, when $\alpha = 2$ this ``bifurcation'' occurs for $a=4$, see
    \sref{sec:results}.
\end{remark}

\begin{remark}
    Compare the Lorenz Renormalization Conjecture with the classical systems of
    unimodal maps, critical circle maps, etc.
    In these systems only case~\ref{T-rigid} can occur and the limit set of
    renormalization, $\mathcal A$, is a \defn{horseshoe};
    that is, $\mathcal A$ is hyperbolic and the restriction $\Rop|\mathcal A$
    is conjugate to a full shift on infinitely many symbols.
    Furthermore, orbits of the renormalization operator (where defined) are
    exponentially contracted to $\mathcal A$ \cite{AL11}.

    As a counterpoint, the limit set of Lorenz renormalization cannot be a
    horseshoe due to case~\ref{T-stratified}; instead, it seems to strictly
    contain a horseshoe which because of case~\ref{T-foliated} does not attract
    all orbits of renormalization.
\end{remark}

\begin{remark}
    Consider how the Lorenz Renormalization Conjecture influences
    \defn{parameter universality} phenomena.

    Classically, a topologically full family (of dimension one) transversally
    intersects a stable manifold (of codimension one) of a hyperbolic
    renormalization fixed point;
    this causes iterated images of the family under renormalization to
    accumulate on an unstable manifold and the bifurcation patterns of the
    family asymptotically look like those of the unstable manifold.

    Here, the iterated images of a topologically full family (which has
    dimension two) under renormalization need not accumulate on an unstable
    manifold;
    it depends on which rigidity class the family hits (a stable manifold may
    have codimension three inside~$\Lset$).
    However, a three-dimensional family will generically hit all rigidity
    classes and hence asymptotically contain all possible bifurcation patterns.
    Universality persists but in a more intricate fashion and there is now a
    distinction between topologically full families (of dimension two) and
    geometrically full families (of dimension three).
\end{remark}

\begin{figure}
    \center
    \begin{tikzpicture}[baseline]
    \begin{axis}[
        unit vector ratio=1 1 1,
        footnotesize,
        title={$f_\flat = \Rop f_\sharp$},
        xtick distance=1,
        ytick=\empty,
        xmin=0, xmax=1,
        ymin=0, ymax=1,
    ]
        \addplot [blue] table {period2-8-2-left.dat};
    \end{axis}
    \end{tikzpicture}
    \begin{tikzpicture}[baseline]
    \begin{axis}[
        unit vector ratio=1 1 1,
        footnotesize,
        title={$f_\star = \Rop f_\star$},
        xtick={0,0.142709,1},
        ytick=\empty,
        xmin=0, xmax=1,
        ymin=0, ymax=1,
    ]
        \addplot [blue] table {fixedpt-8-2.dat};
    \end{axis}
    \end{tikzpicture}
    \begin{tikzpicture}[baseline]
    \begin{axis}[
        unit vector ratio=1 1 1,
        footnotesize,
        title={$f_\sharp = \Rop^2 f_\sharp$},
        xtick={0,0.752928,1},
        ytick=\empty,
        yticklabel pos=upper,
        xmin=0, xmax=1,
        ymin=0, ymax=1,
    ]
        \addplot [blue] table {period2-8-2-right.dat};
    \end{axis}
    \end{tikzpicture}
    \caption{The fixed point $f_\star$ and period--$2$ orbit
        $\{f_\flat,f_\sharp\}$ of monotone $(8,2)$--type.
        Note that $f_\flat$ appears to only have one branch because its
        critical point is very close to zero, $c(f_\flat) \approx 0.0013$.}
        \label{fig:monotone-8-2}
\end{figure}
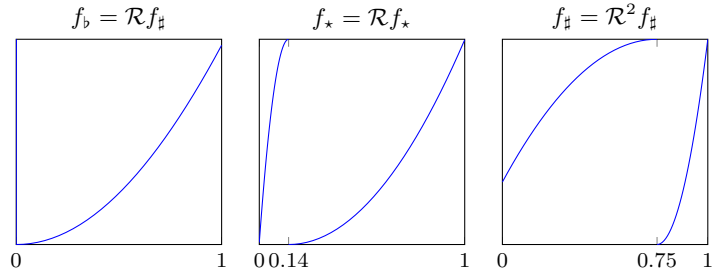

\section{Numerics}

In this section numerical experiments which support the Lorenz Renormalization
Conjecture are described.

The purpose of these experiments is to locate approximate renormalization fixed
points and to estimate the relative sizes of the eigenvalues of the derivative
of~$\Rop$ at these fixed points.
Approximate periodic points of $\Rop$ can also be located with this method by
considering the combinatorics of twice renormalizable maps.
The purpose is \emph{not} to provide accurate estimates.

This method will not rule out existence of other periodic points of
renormalization, only to give evidence in favor of existence of the three cases
of the Lorenz Renormalization Conjecture.
From our observations there seem to be no other cases.

\subsection{Representation of Lorenz maps}

\begin{definition} \label{def:Lrep}
    Let $\diff$ denote the set of orientation-preserving diffeomorphisms
    on~$[0,1]$ and define the family
    \begin{equation*}
        F: (0,1) \times [0,1) \times (0,1] \times \diff \times \diff \to \Lset
    \end{equation*}
    as follows:
    given $(c,v,\phi)$, where $v = (v_0, v_1)$ and
    $\phi = (\phi_0,\phi_1)$, define $F(c,v,\phi)$ to be the Lorenz map
    $f: [0,1]\setminus\{c\} \to [0,1]$ whose branches
    $f_0:[0,c]\to[v_0,1]$ and $f_1:[c,1]\to[0,v_1]$ are the rescalings of
    $\phi_0(1 - (1-x)^\alpha)$ and $\phi_1(x^\alpha)$, respectively (see
    definition~\ref{def:rescale}).
    The parameters $v = (v_0,v_1)$ are called \defn{boundary values}.
\end{definition}

\begin{remark} \label{rem:injective}
    It is clear that $F$ is injective; furthermore, its image is renormalization
    invariant by lemma~\ref{lem:R}.
\end{remark}

\begin{definition} \label{def:Ltrunc}
    Let $\Dtrunc \subset \diff$ be a finite-dimensional subset of
    diffeomorphisms together with a projection $\Dproj: \diff \to \Dtrunc$.
    Let
    \begin{equation*}
        \Ltrunc = (0,1) \times [0,1) \times (0,1] \times \Dtrunc \times \Dtrunc
    \end{equation*}
    denote the \defn{set of truncated Lorenz maps}.
\end{definition}

\begin{remark} \label{rem:diff-approx}
    For simplicity of implementation, we choose $\Dtrunc$ to be a set of
    piecewise linear homeomorphisms.
    Of course, this is not a subset of diffeomorphisms but for the purpose of
    the numerics it empirically does not matter.

    To address the issue of smoothness, cubic interpolation could be used
    instead of linear interpolation, but then care has to be taken that the
    interpolation is monotone.
    Another idea is to linearly interpolate functions on $[0,1]$ and taking the
    inverse of the nonlinearity operator; this would ensure monotonicity as
    well as $\Cset^2$--smoothness.
    A third idea is to use finite pure internal structures, which ensures
    monotonicity and $\Cset^\infty$--smoothness \cite{MW14}.

    We choose not to pursue these paths here as the implementation would become
    more involved and since it would not give qualitatively different results.
\end{remark}

\subsection{Truncated renormalization}

\begin{lemma} \label{lem:R}
    Let $f = F(c,v,\phi)$ as in definition~\ref{def:Lrep}.
    If $f$ is $\word$--renormalizable, then
    $\Rop f = F(c',v',\phi')$ for some $(c',v',\phi')$.
    Explicitly, let $n_k = \abs{\word_k}$, $p_0 = f^{n_1 - 1}(0)$,
    $p_1 = f^{n_0 - 1}(1)$,
    $\tilde\phi_0(x) = v_0 + (1 - v_0)\phi_0(x)$, and
    $\tilde\phi_1(x) = v_1\phi_1(x)$; then
    \begin{equation} \label{renorm-params}
        c' = \frac{c - p_0}{p_1 - p_0},
        \quad
        v_0' = \frac{f^{n_0}(p_0) - p_0}{p_1 - p_0},
        \quad
        v_1' = \frac{f^{n_1}(p_1) - p_0}{p_1 - p_0},
    \end{equation}
    and $\phi'_0$, $\phi'_1$ are the respective rescalings of
    \begin{align*}
        f^{n_0 - 1} \circ \tilde\phi_0:
            [\tilde\phi_0^{-1}\circ f(p_0),1] \to [f^{n_0}(p_0), p_1], \\
        f^{n_1 - 1} \circ \tilde\phi_1:
            [0, \tilde\phi_1^{-1}\circ f(p_1)] \to [p_0, f^{n_1}(p_1)].
    \end{align*}
\end{lemma}

\begin{proof}
    Denote the first-return map associated with the renormalization by
    $g: I\setminus\{c\} \to I$, where $I = [p_0,p_1]$.
    Then $c'$ is the relative position of $c$ in~$I$, $v_0'$ is the relative
    length of $g([p_0,c))$ in~$I$, and $v_1'$ is the relative length of
    $g((c,p_1])$ in~$I$; written out this is \eqref{renorm-params}.
    The statement for $\phi'_0$, $\phi'_1$ is just saying that they are the
    branches of
    $g$ without the initial folding $x^\alpha$ that comes from $f|_I$.
    Since $g$ is a first-return to $I$ the $f$--images of $I$ do not meet
    the critical point before they return; this means that $\phi'_k$ are
    diffeomorphisms.
\end{proof}

\begin{definition} \label{def:Rtrunc}
    Let $F$ and $(\Dtrunc,\Dproj)$ be as in definitions
    \ref{def:Lrep} and \ref{def:Ltrunc}, respectively,
    and let $P(c,v,\phi) = (c,v,\Dproj(\phi_0),\Dproj(\phi_1))$.
    For every renormalizable $F(c,v,\phi)$, define the
    \defn{truncated renormalization operator}, $\Rtrunc$, by
    \begin{equation*}
        \Rtrunc(c,v,\phi) = P \circ F^{-1} \circ \Rop \circ F(c,v,\phi).
    \end{equation*}
    This is well-defined by remark~\ref{rem:injective}.
\end{definition}

\begin{remark}
    For a class of monotone combinatorics with $\abs{\word}$ large the
    renormalization operator is close to having finite dimensional image, in
    the sense that the diffeomorphisms $\phi'_k$ in lemma~\ref{lem:R} are
    close to being linear \cite{MW16}.
    In other words, $\Rtrunc$ can automatically be a good approximation of
    $\Rop$, depending on the combinatorics.
\end{remark}

\begin{remark} \label{rem:renorm3d}
    Taking the above remark to its extreme, it even makes sense to consider
    the trivial set $\Dtrunc = \{ \id \}$ of diffeomorphisms, and looking at
    the corresponding truncated renormalization operator; it is explicitly
    defined by \eqref{renorm-params} with $\phi = (\id,\id)$.
    This is the operator we used to estimate the eigenvalues in
    figure~\ref{fig:eigenvalues}.

    Empirically, it exhibits all the dynamics of the Lorenz
    Renormalization Conjecture and seems to be a remarkably good approximation
    of the full renormalization operator as far as qualitative behavior is
    concerned.
    This should not come as a great surprise as one method of proving existence
    of fixed points for $\Rop$ involves homotoping to this three-dimensional
    truncation and proving it has a fixed point \cites{MW14,MW16}.
\end{remark}

\begin{definition} \label{def:Rmod}
    For every renormalizable $F(c,v,\phi)$,
    define the \defn{modified renormalization operator},
    $\Rmod: (c,v,\phi) \mapsto (c',v')$, in the same way as the
    truncated renormalization operator, except changing \eqref{renorm-params}
    to
    \begin{equation*}
        \begin{aligned}
            c' &= p_0 - c + (p_1 - p_0) c,
            \\
            v_0' &= p_0 - f^{n_0}(p_0) + (p_1 - p_0) v_0,
            \\
            v_1' &= p_0 - f^{n_1}(p_1) + (p_1 - p_1) v_1.
        \end{aligned}
    \end{equation*}
    Note that the image of $\Rmod$ is contained in $\reals^3$.
\end{definition}

\begin{remark}
    The idea of the above operator is to improve the numerical behavior of
    $\Rtrunc$ by not dividing by the length of the return interval
    in~\eqref{renorm-params}.
    From the same equation it can be seen that the set of zeros of~$\Rmod$
    coincide with the set of $(c,v,\phi)$ for which $(c,v)$ are fixed by
    $\Rtrunc$.
    We found that the Newton method on~$\Rmod$ has better convergence properties
    than the Newton method on $\Rtrunc - \id$.
    Given $\word$, we use it to determine what the right value for~$c$ should
    be for a truncated renormalization fixed point (see the fixed point
    algorithm in the next section).
\end{remark}

\subsection{Locating fixed points}
\label{sec:locating}

The perhaps simplest idea for locating fixed points of the truncated
renormalization operator is to use a Newton iteration.
This is feasible for short combinatorics, but for longer combinatorics it
is practically impossible to find starting guesses for which it converges.

The method we employ can be thought of as acting on the two-dimensional
families $v \mapsto F(c,v,\phi)$ (see definition~\ref{def:Lrep}).
It consists of three separate algorithms: one which determines a $v$ such that
$F(c,v,\phi)$ is renormalizable, followed either by an algorithm which takes
$F(c,v,\phi)$ and produces a new $c$, or one which takes $F(c,v,\phi)$ and
produces a new $\phi$.
Combined, these methods empirically behave like a contraction toward a
family which contains a renormalization fixed point and for which the first
algorithm is a contraction toward this fixed point.

\begin{definition}[Renormalization fixed point algorithm]
    Input: the combinatorics $\word$.

    \begin{enumerate}[label=(\arabic*)]
        \item Pick an initial guess for $c$ and $\phi$.
        \item \label{fixedpt-algo-thurston}
            Apply the modified Thurston algorithm to $v \mapsto F(c, v, \phi)$
            to get new boundary values $v'$ (see \sref{sec:thurston} and
            remark~\ref{rem:mod-thurston}).
        \item \label{fixedpt-algo-newton}
            Take a Newton step with the operator $\Rmod$ on $F(c, v', \phi)$ to
            get a new critical point $c'$.
        \item Apply the modified Thurston algorithm to
            $v \mapsto F(c', v, \phi)$ to get new boundary values $v''$.
        \item \label{fixedpt-algo-renorm}
            Apply $\Rtrunc$ to $F(c', v'', \phi)$ to get new diffeomorphisms
            $\phi'$.
        \item \label{fixedpt-algo-last}
            Stop if $(c,v,\phi) = (c',v'',\phi')$, else set $c = c'$,
            $\phi = \phi'$ and go back to step~\ref{fixedpt-algo-thurston}.
    \end{enumerate}
    Output: the Lorenz map $F(c,v,\phi)$ (supposedly a renormalization fixed
    point).
\end{definition}

\begin{remark}
    The above algorithm empirically seems to converge for the initial guesses
    $\phi = (\id,\id)$ and a large set of $c$.
    Theoretically, there is no guarantee for the output to be a renormalization
    fixed point, but practically we observe that it is (as long as the
    algorithm converges).
\end{remark}

\subsection{The Thurston algorithm}
\label{sec:thurston}

The Thurston algorithm is a fixed point method that realizes any periodic
combinatorics in a full family of maps.
It originates in \ocite{DH93} and is also known as the Spider Algorithm in the
complex setting \cite{HS94}.
In real dynamics it is usually employed to prove the full family theorem
\cites{MdM01,dMvS93}.
We use it to locate renormalizable maps within the two-dimensional families
$v \mapsto F(c,v,\phi)$ (see definition~\ref{def:Lrep}).

\begin{definition}[The Thurston Algorithm]
    Input: a critical point $c$, diffeomorphisms $\phi = (\phi_0,\phi_1)$, and
    combinatorics $\word = (\word_0,\word_1)$.

    \begin{enumerate}[label=(\arabic*)]
        \item
            Pick an initial guess of \defn{shadow orbits}\footnote{%
                The name comes from the fact that in the end $x_k$ will be actual
                orbits of the critical values $0$ and~$1$ under some map $f$ in
                the family; i.e.\ $x_k(j) = f^j(k)$.}
            \begin{equation*}
                \{x_k(0) = k,x_k(1),\dots,x_k(m - 1) = c\},\quad
                m = \abs{\word_0} + \abs{\word_1},\;
                k=0,1.
            \end{equation*}
            Let $\altword_k$ be the concatenation of $\word_k$ followed by
            $\word_{1-k}$, for $k=0,1$.
        \item \label{thurston-setv}
            Set $v = (x_0(1),x_1(1))$, and let $f = F(c,v,\phi)$ with branches $f_0$ and~$f_1$.
        \item \label{thurston-pullback}
            Pull back $x_k$ with $f$ according to the combinatorics
            $\altword_k$:
            \begin{equation*}
                y_k(j - 1) = f^{-1}_{\altword_k(j)}(x_k(j)),\quad
                j = 1, \dotsc, m - 1,\;
                k = 0,1.
            \end{equation*}
        \item \label{thurston-setlast}
            Set $y_k(m - 1) = c$, $k=0,1$.
        \item Stop if $y_k = x_k$, else set $x_k$ to $y_k$, $k=0,1$, and go back
            to~\ref{thurston-setv}.
    \end{enumerate}
    Output: the map $f$ which is a realization of the combinatorics $\word$ in
    the family $v \mapsto F(c,v,\phi)$.
\end{definition}

\begin{remark}
    As long as the initial guess is chosen consistently (i.e.\ if the shadow
    orbits are ordered according to $\word$) this algorithm is guaranteed to
    stop; in this case, the realization $f$ is renormalizable and
    the boundary values of $\Rtrunc f$ equal the critical point of $\Rtrunc f$.

    In practice the algorithm converges if the initial guess consists of
    uniformly spaced points $x_0(0) < \dots < x_0(m - 1)$ and
    $x_1(0) > \dots > x_1(m - 1)$ even though these are not ordered according
    to the combinatorics $\word$.
\end{remark}

\begin{remark} \label{rem:mod-thurston}
    We modify the above algorithm so that the realization $f$ fixes its
    boundary values under renormalization; i.e.\ $\Rtrunc f(k) = f(k)$, for
    $k=0,1$.
    This is convenient as we are interested in renormalization fixed points.
    The modification is to replace step~\ref{thurston-setlast} with:
    \begin{enumerate}[label=($4'$)]
        \item Let $p_0 = x_0(\abs{\word_1} - 1)$ and
            $p_1 = x_1(\abs{\word_0} - 1)$ and set
            \begin{equation*}
                y_k(m - 1) = p_0 + (p_1 - p_0) v_k,\quad
                k=0,1.
            \end{equation*}
    \end{enumerate}
    Note that $[p_0, p_1]$ is the return interval of $f$ if $y_k = x_k$, so
    what this step does is to set the relative boundary values of the
    first-return map.
    Replacing $v_k$ with parameters $t_k$ varying in $[0,1]$ it is possible to
    find the whole domain of $\word$--renormalizability in the family.
\end{remark}

\begin{remark}
    There is a relationship between the modified Thurston algorithm from the
    previous remark and the renormalization operator---if the modified Thurston
    algorithm is applied to a family which contains a renormalization fixed
    point then the output of the algorithm will be the renormalization fixed
    point.
    So the renormalization fixed point is also the fixed point of a contractive
    ``Thurston operator.''
\end{remark}

\subsection{Implementation}

The source code for an implementation of the fixed point algorithm of
\sref{sec:locating} is freely available online \cite{W18}.
It compiles to three executables which were used to produce the results of
\sref{sec:results}; see the accompanying README for instructions on how to
reproduce the results.

The Eigen library \cite{GJ10} is used for linear equation solvers and eigenvalue
estimation; we also use its bindings to the multiple precision library MPFR
\cites{FHLPZ07,H08} as well as its automatic differentiation routines.
Standard double precision arithmetic is only sufficient for short
combinatorics, which is why the implementation needs multiple precision.
Automatic differentiation is used to evaluate the derivative of $\Rtrunc$.
Note that this is not the same thing as numerical differentiation (taking
finite differences); instead it uses the chain-rule to exactly (up to numerical
precision) evaluate derivatives.

\subsection{Results}
\label{sec:results}

The experiments in this section were performed using a truncation of
$\Rtrunc$ in dimension three up to dimension $1000$.
Higher dimensions were needed only when evaluating the renormalization of
period--$2$ points, such as in figure~\ref{fig:monotone-8-2}, otherwise the
three-dimensional truncation gave qualitatively accurate results.
Results are only stated for monotone combinatorics; some non-monotone
combinatorics were tested as well but it is harder to present these in a clear
manner so they are not included.
The programs also work with arbitrary $\alpha$ but experiments investigating
the $\alpha$--dependence have been left out to keep this section focused.

The following table shows which of case \ref{T-rigid}, \ref{T-foliated} or
\ref{T-stratified} of the Lorenz Renormalization Conjecture the first few
monotone $(a,b)$--types fall under for $\alpha=2$:

\begin{center}
\footnotesize
\begin{tabular}{cccccccccc|c}
    $\mathbf1$ & $\mathbf2$ & $\mathbf3$ & $\mathbf4$ & $\mathbf5$ & $\mathbf6$ & $\mathbf7$ & $\mathbf8$ & $\mathbf9$ & $\cdots$ & $(a,b)$ \\
    \hline
    A & A & A & A & A & A & A & A & A & $\cdots$ & $\mathbf1$ \\
      & A & A & A & A & A & C & C & C & $\cdots$ & $\mathbf2$ \\
      &   & A & B & B & B & B & B & B & $\cdots$ & $\mathbf3$ \\
      &   &   & B & B & B & B & B & B & $\cdots$ & $\mathbf4$ \\
      &   &   &   & B & B & B & B & B & $\cdots$ & $\mathbf5$ \\
      &   &   &   &   & B & B & B & B & $\cdots$ & $\mathbf6$ \\
      &   &   &   &   &   & B & B & B & $\cdots$ & $\mathbf7$ \\
      &   &   &   &   &   &   & B & B & $\cdots$ & $\mathbf8$ \\
      &   &   &   &   &   &   &   & B & $\cdots$ & $\mathbf9$
\end{tabular}
\end{center}

For example, the above table shows that $(a,a)$--type has a two-dimensional
unstable manifold for $a=1,2,3$, and a three-dimensional unstable manifold for
$a\geq4$; $(a,2)$--types with $a\geq7$ has both a fixed point and a period--$2$
point.
Note that the complete table is symmetric about the diagonal.

\begin{remark}
    It is known that $a$ and $b$ sufficiently large implies
    case~\ref{T-foliated} \cite{MW16}.
    It is not clear exactly when case~\ref{T-stratified} occurs;
    from the above table only $(a,1)$--type and $(a,2)$--type seem viable, but
    a test with increasing $a$ did not reveal any $(a,1)$--types of
    case~\ref{T-stratified}.
    Note that we are only discussing stationary combinatorics and $\alpha=2$
    here.
\end{remark}

In creating the above table we performed roughly the following steps:
\begin{enumerate}[label=(\arabic*)]
    \item Locate a fixed point for the three-dimensional truncated
        renormalization operator (see remark~\ref{rem:renorm3d}), using $c=0.5$
        as an initial guess for the critical point; if it doesn't converge, try
        other values for $c$ until it does.

        The derivative of the three-dimensional truncation of $\Rtrunc$ at the
        fixed point has three eigenvalues.
        Denote the eigenvalue with the smallest magnitude by~$\lambda_c$;
        this is the eigenvalue associated with moving the critical point (the
        other two eigenvalues are associated with changing the boundary
        values).
        If $\lambda_c \in (0,1)$ then we must be in case~\ref{T-rigid};
        if $\lambda_c \in (-1,0]$ we go to the next step;
        if $\abs{\lambda_c} > 1$ we must be in case~\ref{T-foliated}.
        The behavior of $\lambda_c$ is illustrated in
        figure~\ref{fig:eigenvalues}.
    \item Try to locate a period--$2$ orbit of $\Rtrunc$ by looking for a fixed
        point of twice $(a,b)$--renormalizable type.\footnote{%
            For example, once $(2,1)$--renormalizable type is given by
            $(011,10)$ and twice $(2,1)$--renormalizable is given by
            $(0111010,10011)$.}
        We observe in this situation that one of three things happen:
        \begin{enumerate}[label=(\roman*)]
            \item the algorithm diverges by $c\uparrow 1$ (most common case),
            \item the algorithm converges to the fixed point found in the
                previous step (only seems to happen if $c$ is picked close to
                the $c$ of the fixed point),
            \item the algorithm converges and $c$ is different from that of the
                fixed point.
        \end{enumerate}
        In the first two situations we are in case~\ref{T-rigid} and in the
        last situation we are in case~\ref{T-stratified}.
        In the first two situations this step is repeated with different
        guesses for $c$ to make sure the last situation was not missed due to a
        bad initial guess.

        The graphs of the fixed point and period--$2$ orbit for $(8,2)$--type
        can be found in figure~\ref{fig:monotone-8-2}.
    \item Increase the dimension of the truncation of $\Rtrunc$ to see if it
        affects the above classification; in all cases we tried the eigenvalues
        changed slightly in value but not enough to affect the classification.
\end{enumerate}

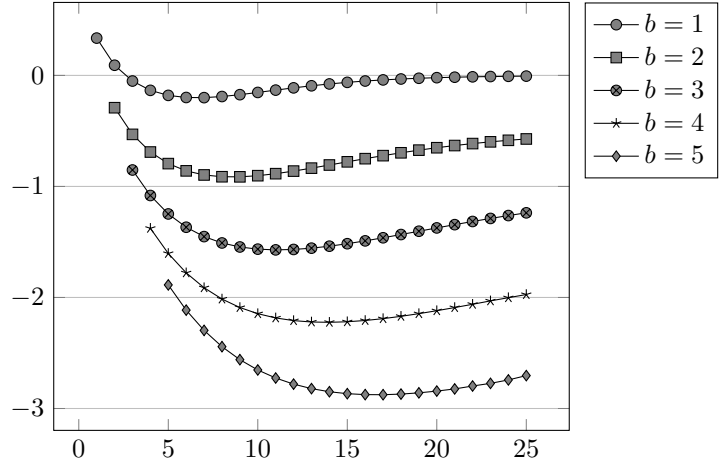
\begin{figure}
    \center
    \begin{tikzpicture}
    \begin{axis}[
        ymajorgrids=true,
        legend entries={$b=1$,$b=2$,$b=3$,$b=4$,$b=5$},
        legend pos=outer north east,
        cycle list name=black white,
    ]
        \addplot table [x=a,y=eval2] {eigenrow-b1.dat};
        \addplot table [x=a,y=eval2] {eigenrow-b2.dat};
        \addplot table [x=a,y=eval2] {eigenrow-b3.dat};
        \addplot table [x=a,y=eval2] {eigenrow-b4.dat};
        \addplot table [x=a,y=eval2] {eigenrow-b5.dat};
    \end{axis}
    \end{tikzpicture}
    \caption{Dependence of the eigenvalue associated with movement of the
        critical point on monotone type $(a,b)$ for $\alpha=2$; estimated using
        the three-dimensional truncation of $\Rtrunc$.}
    \label{fig:eigenvalues}
\end{figure}

\section*{Notation}

\begin{center}
    \begin{tabular}{ p{2cm} p{7cm} p{1cm} }
    $f$, $f_0$, $f_1$
        & Lorenz map $f$ with branches $f_0$, $f_1$
        & \pageref{def:lorenz} \\
    $c$, $c(f)$
        & the critical point of $f$
        & \pageref{def:lorenz} \\
    $\alpha$
        & critical exponent
        & \pageref{def:lorenz} \\
    $\Lset$
        & set of Lorenz maps
        & \pageref{def:lorenz} \\
    $\word = (\word_0, \word_1)$
        & type of renormalization
        & \pageref{def:renorm} \\
    $\Rop$
        & renormalization operator
        & \pageref{def:Rop} \\
    $\Tset_\word$
        & topological class
        & \pageref{conj:Tset} \\
    $v = (v_0,v_1)$
        & boundary values, $v_k = f(k)$
        & \pageref{def:Lrep} \\
    $\phi = (\phi_0,\phi_1)$
        & diffeomorphisms
        & \pageref{def:Lrep} \\
    $F$
        & family of Lorenz maps $F(c,v,\phi)$
        & \pageref{def:Lrep} \\
    $\Dtrunc$, $\Dproj$
        & finite-dimensional diffeomorphism, projection
        & \pageref{def:Ltrunc} \\
    $\Ltrunc$
        & set of truncated Lorenz maps
        & \pageref{def:Ltrunc} \\
    $\Rtrunc$, $\Rmod$
        & truncated renormalization operators
        & \pageref{def:Rtrunc}, \pageref{def:Rmod}
\end{tabular}
\end{center}

\end{document}

%% file: macros.tex
\theoremstyle{plain}
\newtheorem{conjecture}{Conjecture}[section]
\newtheorem*{lrc}{The Lorenz Renormalization Conjecture}
\newtheorem{lemma}[conjecture]{Lemma}
\theoremstyle{definition}
\newtheorem{definition}[conjecture]{Definition}
\newtheorem*{convention}{Convention}
\theoremstyle{remark}

\newtheorem{remark}[conjecture]{Remark}

\DeclareMathOperator{\id}{id}

\newcommand\diff{\mathcal D}
\newcommand\Dtrunc{D}
\newcommand\Dproj{\mathrm{proj}_D}
\newcommand\defn[1]{\textbf{#1}}
\newcommand\Lset{\mathcal L}
\newcommand\Ltrunc{L}
\newcommand\Rop{\mathcal R}
\newcommand\Rtrunc{R}
\newcommand\Rmod{\tilde R}
\newcommand\reals{\mathbb{R}}
\newcommand\sref[1]{\S\,\ref{#1}}
\newcommand\word{w}
\newcommand\altword{W}
\newcommand\Cset{C}
\newcommand\Tset{\mathcal{T}}
\newcommand\cantor{\mathcal{O}}

\DeclarePairedDelimiter\abs{\lvert}{\rvert}